\documentclass[12pt]{amsart}
\usepackage[frenchb, english]{babel}
\usepackage{amssymb,amsfonts,amsmath}
\usepackage{mathrsfs}
\usepackage{anysize}
\usepackage{url}
\usepackage[leqno]{amsmath}
\usepackage{enumitem}
\usepackage{tikz-cd}
\usepackage{stackengine}
\usetikzlibrary{decorations.pathreplacing}

\usepackage{tikz-cd}
\usetikzlibrary{arrows}
\usetikzlibrary{patterns}
\usepackage{bbding}
\usepackage{datetime}
\usepackage{pgfplots}

\usepackage{caption}
\usepackage{tikz}
\usepackage{accents}
\usetikzlibrary{calc}
\usetikzlibrary{arrows}
\usetikzlibrary{backgrounds}
\usepackage{3dplot}
\tdplotsetmaincoords{50}{135}
\tdplotsetrotatedcoords{0}{0}{0} 

\theoremstyle{plain}
\newtheorem{theorem}{Theorem}[section]
\newtheorem{lemma}[theorem]{Lemma}
\newtheorem{corollary}[theorem]{Corollary}

\theoremstyle{definition}
\newtheorem{definition}[theorem]{Definition}
\theoremstyle{remark}

\usetikzlibrary{decorations.markings}
\tikzset{negated/.style={
        decoration={markings,
            mark= at position 0.5 with {
                \node[transform shape] (tempnode) {${\scriptstyle\setminus} $};
            }
        },
        postaction={decorate}
    }
}

\tikzset{degil/.style={
            decoration={markings,
            mark= at position 0.5 with {
                  \node[transform shape] (tempnode) {$\backslash$};
                  }
              },
              postaction={decorate}
}
}

\title[Switched server systems whose parameters are  normal numbers in base $4$]{Switched server systems whose  parameters \\ are normal numbers in base $4$}

\author{Andr\'e do Amaral Antunes}
\address{Departamento de Computa\c c\~ao e Matem\'atica, Faculdade de Filosofia,
	Ci\^encias e Letras, Universidade de S\~ao Paulo, Ribeir\~ao Preto, SP,
	14040-901, Brazil.}
\email{antunes.andre@usp.br}

\author{Yann Bugeaud}
\address{ Institut de Recherche Math\'ematique Avanc\'ee, U.M.R. 7501,
	Universit\'e de Strasbourg et C.N.R.S.,
	7, rue Ren\'e Descartes,
	67084 Strasbourg, France}
\email{bugeaud@math.unistra.fr}

\author{Benito Pires}
\address{Departamento de Computa\c c\~ao e Matem\'atica, Faculdade de Filosofia,
	Ci\^encias e Letras, Universidade de S\~ao Paulo, Ribeir\~ao Preto, SP,
	14040-901, Brazil.}
\email{benito@usp.br}

\date{}


\begin{document}

\maketitle


\marginsize{2.5cm}{2.5cm}{1cm}{2cm}

\begin{abstract} Switched server systems are mathematical models of manufacturing, traffic and queueing systems. Recently, it was proved in (Eur. J. Appl. Math. 31(4) (2020), pp. 682-708) that there exist switched server systems with $3$ buffers (tanks), a server, filling rates $\rho_1=\rho_2=\rho_3=\frac13$ and parameters $d_1, d_2, d_3>0$ whose global attractor is a fractal set. In this article, we prove that if 
$x_1$ in $(0,\frac13)$, $x_2$ in $(\frac13,\frac23)$ and $x_3$ in $(\frac23,1)$   
are rational numbers or normal numbers in base $4$ (or more generally, rich numbers to base $4$) and $(d_1,d_2,d_3)$ is the vector with positive entries satisfying
$$d_1=\frac{1}{3x_1}-1,\quad d_2=\frac{2-3x_2}{3x_2-1}, \quad d_3=\frac{3-3x_3}{3x_3-2},$$
then the corresponding switched server has no fractal attractor. More precisely, the Poincar\'e map of the system has a finite global attractor. The approach we use is to study the topological dynamics of a family of piecewise $\lambda$-affine contractions that includes the Poincar\'e map of the switched server system as a particular case.

     \end{abstract}
 
  

\maketitle

\section{Introduction}

This article contains two types of contributions to the area of dynamical systems. The first contribution is in the research topic named one-dimensional dynamics, dynamics of interval maps with gaps or dynamics of piecewise smooth interval maps. In this regard, we study the topological dynamics of $n$-interval piecewise $\pm\frac{1}{\beta}$-affine contractions  \linebreak $f:[0,1)\to [0,1)$ considering arbitrary integers $\beta, n\ge 2$. Our main contribution here is an explicit condition on the parameters of $f$ which makes it possible to construct concrete examples with a finite $\omega$-limit set $\omega(f)$. The second contribution is an application of a special case of the first result (with $\beta=2$ and $n=4$) to understand the dynamics of switched server systems with $3$ buffers and $1$ server, which are mathematical models of manufacturing, traffic and queueing systems. We provide more details along the specialised forthcoming subsections.

\subsection{Piecewise $\lambda$-affine contractions}

Throughout this article, we denote  by  $f:I\to I$ a self-map of the interval $I=[0,1)$.  Given $-1<\lambda<1$,  we say that  $f$ is an \textit{$n$-interval piecewise $\lambda$-affine contraction} if there exist points $0=x_0<x_1<\cdots<x_{n-1}<x_n=1$ and real numbers $a_1,a_2,\ldots,a_n$ such that
\begin{equation}\label{npc}
f(x)=\lambda x + a_i\quad\textrm{for all}\quad x\in \big[x_{i-1},x_i\big)\quad\textrm{and}\quad 1\le i\le n.
\end{equation} 
 We call the non-empty set
\begin{equation}\label{olset}
 \omega(f)=\bigcup_{x\in I}\omega(x,f),\quad\textrm{where}\quad\omega(x,f)=\bigcap_{r\ge 0}\overline{\bigcup_{k\ge r}\{ f^k(x)\}},
\end{equation}
 the \textit{$\omega$-limit set} or \textit{global attractor of $f$}.
 
 An important class of piecewise $\lambda$-affine contractions arises from the self-map \linebreak $f:\mathbb{T}^1\to \mathbb{T}^1$ of the  circle $\mathbb{T}^1$ defined by $f(x)=\{\lambda x + a\}$, where $0<\lambda<1$, $0<a<1$ and $\{\cdot\}$ denotes the fractional part of a real number. Writing $f$ as an interval map leads to the $2$-interval piecewise $\lambda$-affine contraction $f:I\to I$ defined by
 \begin{equation}\label{2i}
 f(x)=\begin{cases} \lambda x + a & \textrm{if} \quad x\in \left[0, \frac{1-a}{\lambda} \right)\\[0.1em]
  \lambda x + a - 1 & \textrm{if} \quad x\in \left[ \frac{1-a}{\lambda},1 \right).   
 \end{cases}
 \end{equation}
 
 The topological dynamics of the $2$-interval piecewise $\lambda$-affine contractions given by 
  \eqref{2i} was studied by many authors by means of a rotation number approach and Farey trees (see \cite{MR1240802,MR1683629,MR4030545,MR3815128}). In particular, it is known that there exists a dichotomy: the global attractor of $f$ is either a finite set (generic case) or a Cantor set. The study of the global attractor of a  piecewise (affine or not affine) contraction $f$ defined on the union $E_0\cup E_1$ of two complete metric spaces and such that $f\vert_{E_0}$ and $f\vert_{E_1}$    
  are contractions, was accomplished in \cite{MR1018928}.
  
 The topological dynamics of $n$-interval piecewise $\lambda$-affine contractions, with $n\ge 2$, was considered in \cite{MR3820005} and in \cite{p1}. In contrast to the case  $n=2$ where a dichotomy is present, when $n\ge 3$ the global attractor can be of three types: a finite set (generic case), finitely many Cantor sets, or the union of a finite set and finitely many Cantor sets (see \cite[Theorem 1.1]{ACECOG2020}). Piecewise $\lambda$-affine maps with Cantor attractors were constructed in \cite[Corollary 2.5 and Theorem 2.7]{MR4030596} and in \cite{MR4120256}. It is also worth mentioning that every injective piecewise (affine or not affine) contraction is topologically conjugate to a piecewise affine contraction (see \cite[Theorem 1.2]{MR3394114}).
 
 Although the results provided in \cite{p1,MR3225875,MR3820005} are far comprehensive, they are not appropriate for concrete applications since it is not possible to exhibit parameters that lie in the full measure set of generic parameters. In this respect, the novelty and advantage of our result is that we can provide explicit (rational and irrational) parameters $\lambda$, $a_i$ and $x_i$ of $f$ in \eqref{npc} that make $\omega(f)$ finite. The approach we use to prove the results connects in a natural way piecewise contractions to the fascinating world of  $\beta$-transformations. 
 
\subsection{Switched server systems}\label{sss} 

Switched server systems are mathematical models of manufacturing, traffic and queueing systems. We consider here 
the same model studied in \cite{MR4120256}, which turns out to be a generalisation of the model introduced by Chase et al. in 
\cite[Section II.B, p. 72]{MR1201496}.            
It consists of three buffers (tanks) numbered 1, 2, 3, and a server. At each time $t \ge 0$, a fluid  is delivered to each tank $i$ at the constant rate $\rho_i=\frac13$ ($i=1,2,3$) and is removed from a selected 
tank $j$ in $\{1, 2, 3\}$   
by the server at the constant rate $\rho = 1.$ The volume of fluid in the tank $i$ at the time $t$ is denoted by $v_i(t)$. When the tank $i$ is emptied by the server at the time $t$, the server changes its location to the tank $j\neq i$ with the largest scaled volume $d_{ij} v_j(t)$, where $\{d_{ij}: 1\le i,j\le 3, i\neq j\}\subset \mathbb{R}_+^*$ are parameters of the system. It was shown in \cite{MR4120256} that the dynamics of the switched server system depends only on the proportions of pairs of parameters
\begin{equation}\label{d1d2d3}
 \frac{d_{13}}{d_{12}}=d_1, \quad \frac{d_{21}}{d_{23}}=d_2,\quad \frac{d_{32}}{d_{31}}=d_3.
 \end{equation}
We assume that $\sum_{i=1}^3 v_i(0)=1$ so that $\sum_{i=1}^3 v_i(t)=1$ for every $t\ge 0$. In this way, the state $\mathbf{v}(t)=(v_1(t),v_2(t),v_3(t))$ of the system at the time $t\ge 0$ is a  vector in the phase space 
$$    
\Delta=\{\mathbf{v}=(v_1,v_2,v_3): v_1, v_2, v_3 \ge 0 \,\textrm{and}\,\, v_1 +v_2+v_3=1 \}. 
$$
Let $l(t)$ denote the position of the server at the time $t$. We assume that 
the map $t\mapsto l(t)$ is right-continuous. 

It follows from \cite[Theorem 1.4]{MR3820005} and from \cite[Equation (1.2)]{MR4120256} that, for Lebesgue almost every vector $(d_1,d_2,d_3)$ with positive entries, any switched server system with parameters $d_{ij}$ satisfying \eqref{d1d2d3} is structurally stable and admits finitely many limit cycles that attract all the orbits. However, the set of generic parameters $(d_1,d_2,d_3)$ is not computable, i.e., it is not possible (or at least not easy) to know if a previously chosen positive vector $(d_1,d_2,d_2)$ lies in the generic set of parameters. In \cite[Theorem 2.2]{MR4120256} it was provided a non-generic positive vector $(d_1,d_2,d_2)$ such that any switched server system with parameters $d_{ij}$ satisfying \eqref{d1d2d3} admits a fractal global attractor.

\subsection{The interplay between switched server systems and interval dynamics}\label{int}

 The dynamics of the switched server system described in Section \ref{sss} with $\rho_1=\rho_2=\rho_3=\frac13$ is governed by an injective piecewise $-\frac12$-affine contraction. More precisely, the dynamics of the switched server system is completely determined by the Poincar\'e map $F:\partial\Delta\to\partial\Delta$ induced by the flow 
 $$       
 t\in [0,\infty)\mapsto \mathbf{v}(t)\in\Delta       
 $$
 on the boundary $\partial\Delta$ of the phase space. The Poincar\'e map $F$ is topologically conjugate to the  interval map $f:I\to I$ defined by $f=\phi^{-1}\circ F\circ\phi$, where $\phi:[0,1)\to\partial\Delta$ is the anticlockwise arc-length parametrisation of $\partial\Delta$ with
 $\phi(0)=\mathbf{e}_2=(0,1,0)$, $\phi(\frac13)=\mathbf{e}_3=(0,0,1)$ and $\phi(\frac23)=\mathbf{e}_1=(1,0,0)$. The map $f$, computed in \cite[Equation (1.2)]{MR4120256}, is given by
\begin{equation}\label{fd1d2d3}
f(x)=f_{d_1,d_2,d_3}(x)=
\begin{cases} -\dfrac12 x + \dfrac12 & \textrm{if} \,\, x\in [x_0,x_1)\\[0.5em]
-\dfrac12 x + 1 & \textrm{if} \,\, x\in [x_1,x_2)\\[0.5em]
-\dfrac12 x + \dfrac12 & \textrm{if} \,\, z\in [x_2,x_3) \\[0.5em]
-\dfrac12 x + 1 & \textrm{if} \,\, x\in [x_3,x_4), \\[0.5em]     
\end{cases}
\end{equation}
where
\begin{equation}\label{y123} 
x_0=0,\quad 
x_1=\dfrac{d_2}{3(d_2+d_3)},\quad x_2=\dfrac{d_3}{3(d_1+d_3)}+\dfrac{1}{3},\quad
x_3=\dfrac{d_1}{3(d_1+d_2)}+\dfrac{2}{3}, \quad x_4=1.
\end{equation}

Figure \ref{thesss} shows a graphical representation of the interplay between the switched server system, the Poincar\'e map $F:\partial\Delta\to\partial\Delta$ and the piecewise $-\frac12$-affine contraction $f:I\to I$ considering the parameter values $d_1=d_2=d_3=1$ and $d_{ij}=1$ for all $i\neq j$.

 \begin{figure}[!htb]\vspace{-0.6cm}
   \begin{minipage}{0.3\textwidth}
      \centering
       \vspace*{1cm}
     \begin{tikzpicture}[scale=1]
\begin{scope}[shift={(0,0)}, scale=0.5]
    \draw[very thick] (0,6)--(0,0.5)--(1.3,0.5)--(1.3,0)--(1.7,0)--(1.7,0.5)--(3,0.5)--(3,6);
    \begin{pgfonlayer}{background}
        \filldraw[blue!25] (0,3.5)--(3,3.5)--(3,0.5)--(1.7,0.5)--(1.7,0)--(1.3,0)--(1.3,0.5)--(0,0.5)--cycle;
     \end{pgfonlayer}
      \draw[very thick,->] (1.5,6.5)--(1.5,5.5)node[pos=0,anchor=south west]{$\rho_1=\frac13$};
        \node at (1.4,1){Tank 1};
\draw[very thick,->] (1.5,0)--(1.5,-2)--(5,-2);
    \draw[very thick] (4,6)--(4,0.5)--(5.3,0.5)--(5.3,0)--(5.7,0)--(5.7,0.5)--(6,0.5)--(7,0.5)--(7,6);
    \begin{pgfonlayer}{background}
    \filldraw[blue!25] (4,2.5)--(4,0.5)--(5.3,0.5)--(5.3,0)--(5.7,0)--(5.7,0.5)--(6,0.5)--(7,0.5)--(7,2.5)--cycle;
     \end{pgfonlayer}
      \draw[very thick,->] (5.5,6.5)--(5.5,5.5)node[pos=0,anchor=south west]{$\rho_2=\frac13$};
      \node at (5.5,1){Tank 2};
       \draw[very thick,dashed,->] (5.5,0)--(5.5,-0.5)--(5.5,-1.5)node[pos=0,anchor=south west]{};
    \draw[very thick] (8,6)--(8,0.5)--(9.3,0.5)--(9.3,0)--(9.7,0)--(9.7,0.5)--(11,0.5)--(11,6);
    \begin{pgfonlayer}{background}
    \filldraw[blue!25] (8,4.5)--(8,0.5)--(9.3,0.5)--(9.3,0)--(9.7,0)--(9.7,0.5)--(11,0.5)--(11,4.5)--cycle;
     \end{pgfonlayer}
      \draw[very thick,->] (9.5,6.5)--(9.5,5.5)node[pos=0,anchor=south west]{$\rho_3=\frac13$};
      \node at (9.5,1){Tank 3};
       \draw[very thick,dashed,->] (9.5,-0.2)--(9.5,-2)--(6,-2)node[pos=0,anchor=south west]{};
       \draw[very thick,->] (5.5,-2.5)--(5.5,-3.5) node[xshift=-1cm, yshift=0.45cm]{$\rho=1$};
      
      
     \node [draw,circle, minimum width=0.2 cm,very thick](B) at (5.5,-2){}; 
     \draw[very thick, ->](-2,-3.5)--(-2,7);
      \draw[very thick,dashed](-2.1,3.5)--(3,3.5) node[pos=0,left]{$v_1$};
      \draw[very thick,dashed](-2.1,2.5)--(7,2.5) node[pos=0,left]{$v_2$};
       \draw[very thick,dashed](-2.1,4.5)--(11,4.5) node[pos=0,left]{$v_3$};
        \draw[very thick,dashed](-2.1,0.5)--(11,0.5) node[pos=0,left]{$0$};

\end{scope}
\end{tikzpicture}
   \end{minipage}\hfill
   \begin {minipage}[m]{0.27\textwidth}
     \centering
      \vspace*{1.25cm}

\begin{tikzpicture}[tdplot_main_coords, scale=3.2]
\begin{scope}[shift={(0,1)}]
   \draw [,fill opacity=0.5]
          (1,0,0) -- (0,1,0) -- (0,0,1) -- cycle;
   \draw (1,0.2) node[left] {$\mathbf{e}_1$};
   \draw (0.1,1,-0.05) node[right] {$\mathbf{e}_2$};
   \draw (0.03,-0.1,1.03) node[right] {$\mathbf{e}_3$};
   \draw (0.03,-0.1,-0.75) node[right] {$F(\mathbf{e}_3)$};
 
\begin{scope}[blue, very thick,decoration={
    markings,
    mark=at position 0.65 with {\arrow[blue,scale=1.2]{latex}}}
    ] 
\draw[postaction={decorate}] ( 0, 0.666666666666667, 0.333333333333333 ) -- ( 0.333333333333333, 0, 0.666666666666667 ); 
\draw[postaction={decorate}] ( 0.333333333333333, 0, 0.666666666666667 ) -- ( 0.666666666666667, 0.333333333333333, 0 ); 
\draw[postaction={decorate}] ( 0.666666666666667, 0.333333333333333, 0 ) -- ( 0, 0.666666666666667, 0.333333333333333 ); 
 \end{scope}
\begin{scope}[blue, thin,decoration={
    markings,
    mark=at position 0.65 with {\arrow[blue,scale=1.2]{latex}}}
    ] 

\draw[postaction={decorate}] ( 0.333333333333333, 0, 0.666666666666667 ) -- ( 0.666666666666667, 0.333333333333333, 0 ); 
\draw[postaction={decorate}] ( 0.666666666666667, 0.333333333333333, 0 ) -- ( 0, 0.666666666666667, 0.333333333333333 ); 
\draw[postaction={decorate}] ( 0, 0.666666666666667, 0.333333333333333 ) -- ( 0.333333333333333, 0, 0.666666666666667 ); 
\end{scope}
\begin{scope}[red ,thick, decoration={
    markings,
    mark=at position 0.35 with {\arrow[red,scale=1.2]{latex}}}
    ] 

\draw[postaction={decorate}] ( 0, 0, 1 ) -- ( 0.5, 0.5, 0 ); 
\draw[postaction={decorate}] ( 0.5, 0.5, 0 ) -- ( 0, 0.75, 0.25 ); 
\draw[postaction={decorate}] ( 0, 0.75, 0.25 ) -- ( 0.375, 0, 0.625 ); 
\draw[postaction={decorate}] ( 0.375, 0, 0.625 ) -- ( 0.6875, 0.3125, 0 ); 
\end{scope}

    ] 

    
 \end{scope}  
 \end{tikzpicture}
   \end{minipage}
    \begin{minipage}{0.3\textwidth}
     \centering
     \vspace*{1cm}
\begin{tikzpicture}[scale=0.8]



\draw [  thick, ->] (0,0) -- (5.3,0) node [right] {\footnotesize $x$};
\draw [  thick, ->] (0,0) -- (0,5.3) node [above] {\footnotesize $f(x)$};
	
			

\draw (0,0)--(5,0)--(5,5)--(0,5);
			
\draw[fill=black] (0,5/2) circle (0.1);
\draw[fill=white] (5/6, 25/12) circle (0.1);
\draw[very thick] (0,5/2)--(5/6-0.08, 25/12+0.03) ;

\draw[fill=black] (5/6,55/12) circle (0.1);
\draw[fill=white] (5/2, 15/4) circle (0.1);
\draw[very thick] (5/6,55/12)--(5/2-0.1, 15/4+0.04) ;

\draw[fill=black] (5/2,5/4) circle (0.1);
\draw[fill=white] (25/6, 5/12) circle (0.1);
\draw[very thick] (5/2,5/4)--(25/6-0.1, 5/12+0.05) ;

\draw[fill=black] (25/6,35/12) circle (0.1);
\draw[fill=white] (5, 5/2) circle (0.1);
\draw[very thick] (25/6,35/12)--(5-0.1, 5/2+0.05) ;

\draw [  thick, dotted] (5/6,0)--(5/6,5);
\draw [  thick, dotted] (5/2,0)--(5/2,5);
\draw [  thick, dotted] (25/6,0)--(25/6,5);
\draw [  thick, dotted] (0,5/2)--(5,5/2);
\draw [  thick, dotted] (0,5/12)--(5,5/12);
\draw [  thick, dotted] (0,15/4)--(5,15/4);
\draw [  thick, dotted] (0,25/12)--(5,25/12);
\draw [  thick, dotted] (0,55/12)--(5,55/12);

\draw (0.6,0) node[below right] {$\frac16$};
\draw (5/2,0) node[below] {$\frac12$};
\draw (25/6,0) node[below] {$\frac56$};
 
\end{tikzpicture}
   \end{minipage}\hfill
   \caption{The switched server system, the Poincar\'e map $F:\partial\Delta\to\partial\Delta$ and the piecewise $\frac12$-affine contraction $f:I\to I$.}\label{thesss}
\end{figure}
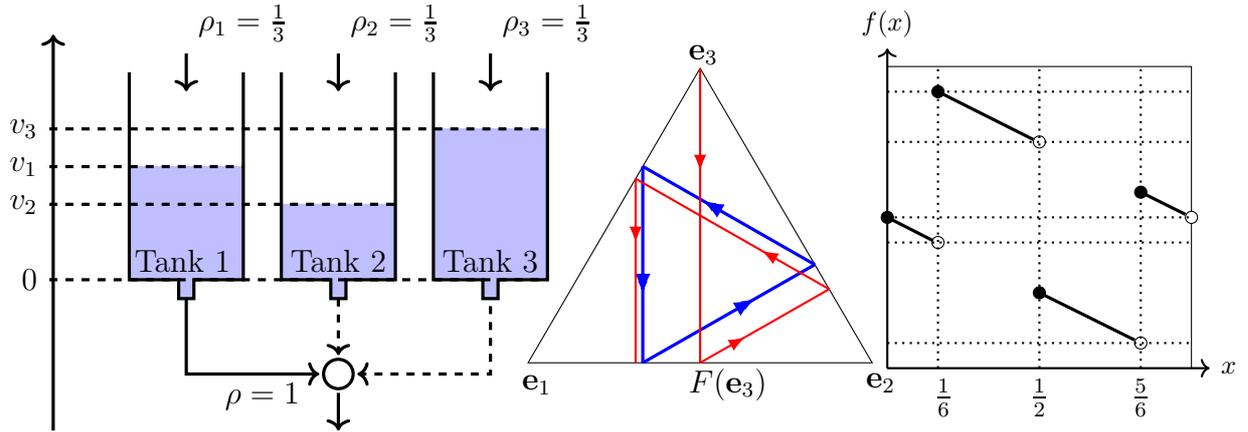

\section{Statement of the results}\label{section2}

We will need some terminology from Symbolic Dynamics and Number Theory. Let $\beta\ge 2$ be an integer and set $\mathcal{A}=\{0,1,\ldots,\beta-1\}$. Given an infinite word $w=w_0 w_1\ldots$ over the alphabet $\mathcal{A}$
and a positive integer $k$,   
we denote by $L_k(w)=\{w_{i} w_{i+1}\ldots w_{i+k-1}: i\ge 0\}$ the set of 
factors     
of $w$ of length $k$. The \textit{factor complexity}    
of $w$ is the map $p_{w}:\mathbb{N}\to\mathbb{N}$ defined by $p_w(k)=\#L_k(w)$, where $\#$ stands for cardinality. To each infinite word
$w=w_0 w_1\ldots$ over the alphabet $\mathcal{A}$, we can
associate  the real number $x$ in $[0, 1]$ whose $\beta$-ary expansion is  $w$.
More precisely, $x=\varphi(w)$,  where  $\varphi: \mathcal{A}^{\mathbb{N}}\to [0,1]$ is the function 
defined by
$$ \varphi(w) =\sum_{i=0}^\infty w_i \beta^{-(i+1)}=\beta^{-1}\sum_{i=0}^\infty w_i \beta^{-i}.$$

We say that $x=\varphi(w)$ is a \textit{rich number to base $\beta$} or a \textit{disjunctive number to base $\beta$}  if $p_w(k)=\beta^k$ for all $k\ge 1$, i.e., if all finite words over $\mathcal{A}$ occur in the $\beta$-ary expansion of $x$ (see \cite[\S 4.4, p. 91]{MR2953186}). Concerning numbers in the interval $[0,1]$, the set $\mathcal{R}_{\beta}$ made up of rich numbers to base $\beta$ contains the proper subset $\mathcal{N}_{\beta}$ of normal numbers in  base $\beta$. In particular,
 $\mathcal{R}_{\beta}$ is a residual $G_{\delta}$-set of Lebesgue measure $1$  (see \cite[Lemma 19]{MR1895809}).  The property of being a rich number depends on the base considered: see for instance the examples constructed in \cite{MR1639533}. Furthermore, there exist numbers that are rich in every base, see the explicit examples provided in \cite{MR2990128}. Anyway, it is much easier to construct a rich number to base $\beta$ than a normal number in the same base.  We will also consider a third class of numbers of $[0,1]$ defined as follows. Given integers $\beta, n\ge 2$, 
 set 
 $$
 \ell = \min \{k\in\mathbb{N}:  2\beta^{-k} < (1-\beta^{-1})/(n+1) \}
 = 1 + \left\lceil {\frac{\log \big(2 (n+1) / (\beta-1)\big)}{\log \beta}} \right\rceil,     
 $$ 
 where $\lceil x \rceil$ denotes the smallest integer strictly greater than $x$,    
 and let $S_{\beta,n}\subset [0,1]$ be the set defined by  
\begin{equation}\label{sbn}
S_{\beta,n}= \{\varphi(w): w\in\mathcal{A}^{\mathbb{N}}\quad \textrm{and}\quad p_w(\ell)=\beta^{\ell}\}.   
\end{equation}
The following nested sequence of inclusions holds for each integer $n \ge 2$:   
$$
 \mathcal{N}_{\beta}\subsetneq \mathcal{R}_{\beta}\subsetneq S_{\beta,n}.
 $$
 Moreover, $S_{\beta,n}$ is a neighbourhood of $\mathcal{R}_{\beta}$ and $I{\setminus} S_{\beta,n}$  
 has Hausdorff dimension less\linebreak than $1$. 
  
Now we state our results concerning the topological dynamics piecewise $\lambda$-affine maps. We denote by $\mathbb{Q}$ the set of rational numbers in $I$.  

\begin{theorem}\label{thm1} Let $\beta,n\ge 2$ be integers, $x_0=0$, $x_n=1$, $0<x_1<\cdots<x_{n-1}<1$ \linebreak be elements of the set $\mathbb{Q}\cup S_{\beta,n}$ defined by $(\ref{sbn})$ and $\alpha_1$, $\alpha_2$, \ldots, $\alpha_{n}$ be elements of the set $\{1,2,\ldots,\beta\}$. Then
 the global attractor $\omega(f)$ of the $n$-interval piecewise $\frac{1}{\beta}$-affine contraction 
$f: I\to I$ defined by $$f(x)=\frac1\beta x + \frac{\alpha_i-1}{\beta}\quad\textrm{for all}\quad x\in [x_{i-1},x_i)\quad\textrm{and}\quad 1\le i\le n,$$ is finite.
\end{theorem}

Given integers $\beta, n\ge 2$,  set
 $$
 \ell' = \min \{k\in\mathbb{N}:  2\beta^{-2k} < (1-\beta^{-1})/(n+1) \}
 =   \left\lceil {\frac12+ \frac{\log \big( 2(n+1)  / (\beta-1)\big)}{2 \log \beta}} \right\rceil       
 $$ 
 and let $S'_{\beta,n}\subset [0,1]$ be the set defined by  
\begin{equation}\label{sbnp}
S'_{\beta,n}= \left\{\varphi(w): w\in\{0,1,\ldots,\beta^2-1\}^{\mathbb{N}}\quad \textrm{and}\quad p_w(\ell')=\beta^{2\ell'}\right\}.   
\end{equation}


The following nested sequence  of inclusions holds for each integer $n\ge 2:$
$$ \mathcal{N}_{\beta^2}\subset \mathcal{R}_{\beta^2}\subset \mathcal{S}_{\beta,n}'. 
$$

\begin{theorem}\label{thm2} Let $\beta,n\ge  2$ be integers, $x_0=0$, $x_n=1$, $0<x_1<\cdots<x_{n-1}<1$ be 
elements of the set {$\mathbb{Q}\cup S'_{\beta,n}$} defined by $(\ref{sbnp})$  
and $\alpha_1$, $\alpha_2$, \ldots, $\alpha_{n}$ be elements of the set $\{1,2,\ldots,\beta\}$ with $\alpha_1\neq \beta$. Then
 the global attractor $\omega(f)$ of the $n$-interval piecewise $-\frac{1}{\beta}$-affine contraction 
$f: I\to I$ defined by $$f(x)=-\frac1\beta x + \frac{\alpha_i}{\beta}\quad\textrm{for all}\quad x\in [x_{i-1},x_i)\quad\textrm{and}\quad 1\le i\le n,$$ is finite.
\end{theorem}

 Notice that taking $\beta=2, n=4$, $\alpha_1=1$, $\alpha_2=2$, $\alpha_3=1$, $\alpha_4=2$ in Theorem \ref{thm2} leads to the map $f_{d_1,d_2,d_3}$ given by \eqref{fd1d2d3}. Moreover, the result in Theorem \ref{thm2} is optimal because in \cite[Definition 6.2, Proof of Proposition 6.4 and Figure 4]{MR4120256} it was constructed an example with $\beta=2$, $\alpha_1=1$, $\alpha_2=2$, $\alpha_3=1$, $\alpha_4=2$ such that the interval map in Theorem \ref{thm2} has a Cantor attractor. In this case, it follows from Theorem \ref{thm2} that
at least one of the numbers $x_i$ (denoted by $z_i$ in \cite{MR4120256}), $1\le i\le 3$, is {an irrational number that is not rich}  to base $4$ (in particular, such $x_i$ is not a normal number in base $4$). More pecisely, since the integer $\ell'$ corresponding to these parameters is equal to $3$, there exists  
at least one word of length $3$ over $\{0, 1, 2, 3\}$ which does not occur in the $4$-ary    
expansion of (at least) one of the $x_i$'s.  


Now we state a consequence of Theorem \ref{thm2} in the dynamics of switched server systems. 

\begin{corollary}\label{cor1} 
Let $x_1$ in $(0,\frac13)$, $x_2$ in $(\frac13,\frac23)$ and $x_3$ in $(\frac23,1)$   
{be  rational numbers or irrational numbers whose $4$-ary expansions contain all words of length $3$ over $\{0,1,2,3\}$}
 and let $(d_1,d_2,d_3)$ be the vector with positive entries satisfying
$$ d_1=\frac{1}{3x_1}-1, \quad d_2=\frac{2-3x_2}{3x_2-1},\quad d_3=\frac{3-3x_3}{3x_3-2}.
$$
Then  any switched server system with parameters $d_{ij}$ satisfying \eqref{d1d2d3} has no fractal attractor or, more precisely, the global attractor $\omega(F)$ of its Poincar\'e map $F:\partial\Delta\to\partial\Delta$ is finite.
\end{corollary}

A more concrete example can be constructed by considering the  base-$4$ Champernowne constant 
$$c=\varphi(\zeta),\quad \textrm{where}\quad \zeta=1\,2\, 3\, 10\, 11\,12\, 13\, 20\, 21\, 22\, 23 \, 30 \, 31 \, 32 \, 33 \, 100\ldots$$ is the infinite word over $\{0,1,2,3\}$  obtained by concatenation of the finite  $4$-ary expansions of the positive integers. Clearly, by construction, $c$ is rich to base $4$.  Notice that $c$ is the real number whose decimal expansion is 
$$0.42611111111111106576455657142016198509554623896723\ldots$$

\begin{corollary}\label{cor2} Let $x_1=c-\frac{1}{4}$, $x_2=c$ and $x_3=c+\frac{1}{2}$,  
where $c$ is the base-$4$ Champernowne constant. Let $(d_1,d_2,d_3)$ be as in Corollary \ref{cor1}.
Then  any switched server system with parameters $d_{ij}$ satisfying \eqref{d1d2d3} has no fractal attractor or, more precisely, the global attractor $\omega(F)$ of its Poincar\'e map $F:\partial\Delta\to\partial\Delta$ is finite.
\end{corollary}

\tdplotsetmaincoords{60}{125}
\tdplotsetrotatedcoords{0}{0}{0} 
\begin{figure}[h!]
 \begin{minipage}{0.48\textwidth}
      \centering
 \begin{tikzpicture}[scale=4,tdplot_rotated_coords,
                    cube/.style={very thick,black},
                   grid/.style={very thin,gray},
                 axis/.style={->,black,  thick},
                    rotated axis/.style={->,purple,  thick}]

    \draw[very thick, axis,tdplot_main_coords] (0,0,0) -- (1.2,0,0) node[anchor=north]{$v_1$};
    \draw[very thick, axis,tdplot_main_coords] (0,0,0) -- (0,1.2,0) node[anchor=north west]{$v_2$};
    \draw[very thick, axis,tdplot_main_coords] (0,0,0) -- (0,0,1.2) node[anchor=west]{$v_3$};

\draw[very thick] (1,0,0)--(0,1,0)--(0,0,1)--cycle;

\begin{scope}[blue, very thick,decoration={
    markings,
    mark=at position 0.75 with {\arrow[blue,scale=1.2]{latex}}}
    ] 
\draw[postaction={decorate}] ( 0.4, 0.6, 0 ) -- ( 0, 0.8, 0.2 ); 
\draw[postaction={decorate}] ( 0, 0.8, 0.2 ) -- (0.4, 0, 0.6 ); 
\draw[postaction={decorate}] ( 0.4, 0, 0.6 ) -- ( 0, 0.2, 0.8 ); 
\draw[postaction={decorate}] ( 0, 0.2, 0.8 )--(0.4, 0.6, 0);
 \end{scope}
\begin{scope}[blue, thin,decoration={
    markings,
    mark=at position 0.65 with {\arrow[blue,scale=1.2]{latex}}}
    ] 
\end{scope}
\begin{scope}[red ,thick, decoration={
    markings,
    mark=at position 0.37 with {\arrow[red,scale=1.2]{latex}}}
    ] 
  \draw[postaction={decorate}] ( 0.11, 0, 0.89 ) -- (0.555, 0.445, 0 ); 
\draw[postaction={decorate}] ( 0.555, 0.445, 0) -- ( 0, 0.7225, 0.2775); 
\draw[postaction={decorate}] ( 0, 0.7225, 0.2775 ) -- ( 0.36125, 0,0.63875 ); 
\draw[postaction={decorate}] ( 0.36125, 0,0.63875 ) -- ( 0, 0.180625, 0.819375 ); 
\end{scope}
\begin{scope}[red ,thick, decoration={
    markings,
    mark=at position 0.8 with {\arrow[red,scale=1.2]{latex}}}
    ] 

\end{scope}
\begin{scope}
 \draw (0.11, 0, 0.89) node[left] {$\mathbf{v}(0)$};
 \path (0.11, 0, 0.89) node[circle, red, fill, inner sep=1]{};
    
 \end{scope}  
 \draw node at (-0.2,0.2,-0.65) {$\mathbf{v}(0)=(0.11,\,0,\, 0.89)$};
 
\end{tikzpicture}
 \end{minipage}
  \begin{minipage}{0.48\textwidth}
  \centering
 \begin{tikzpicture}[scale=4,tdplot_rotated_coords,
                    cube/.style={very thick,black},
                   grid/.style={very thin,gray},
                 axis/.style={->,black,  thick},
                    rotated axis/.style={->,purple,  thick}]

    \draw[very thick, axis,tdplot_main_coords] (0,0,0) -- (1.2,0,0) node[anchor=north]{$v_1$};
    \draw[very thick, axis,tdplot_main_coords] (0,0,0) -- (0,1.2,0) node[anchor=north west]{$v_2$};
    \draw[very thick, axis,tdplot_main_coords] (0,0,0) -- (0,0,1.2) node[anchor=west]{$v_3$};

\draw[very thick] (1,0,0)--(0,1,0)--(0,0,1)--cycle;
  \draw ( 0.8,0,0.2) node[left] {$\mathbf{v}(0)$};
   \path (0.8,0,0.2) node[circle, red, fill, inner sep=1]{};

 \begin{scope}[red ,   thick, decoration={
    markings,
    mark=at position 0.3 with {\arrow[red,scale=1.2]{latex}}}
    ] 
  \draw[postaction={decorate}] ( 0.8,0,0.2) -- ( 0, 0.4, 0.6); 
\draw[postaction={decorate}] ( 0, 0.4, 0.6 ) -- ( 0.3, 0.7, 0 ); 
\draw[postaction={decorate}] ( 0.3, 0.7, 0 ) -- (0, 0.85, 0.15); 
\draw[postaction={decorate}] ( 0, 0.85, 0.15 ) -- (0.425, 0, 0.575); 
\draw[postaction={decorate}] ( 0.425, 0, 0.575 ) -- (0, 0.2125, 0.7875); 

\end{scope}

 \begin{scope}[blue, very thick,decoration={
    markings,
    mark=at position 0.65 with {\arrow[blue,scale=1.2]{latex}}}
    ] 
\draw[postaction={decorate}] ( 0.4, 0.6, 0 ) -- ( 0, 0.8, 0.2 ); 
\draw[postaction={decorate}] ( 0, 0.8, 0.2 ) -- (0.4, 0, 0.6 ); 
\draw[postaction={decorate}] ( 0.4, 0, 0.6 ) -- ( 0, 0.2, 0.8 ); 
\draw[postaction={decorate}] ( 0, 0.2, 0.8 )--(0.4, 0.6, 0);
 \end{scope}
\begin{scope}[blue, thin,decoration={
    markings,
    mark=at position 0.65 with {\arrow[blue,scale=1.2]{latex}}}
    ] 
    
\end{scope}
 \draw node at (-0.2,0.2,-0.65) {$\mathbf{v}(0)=( 0.8,0,0.2)$};
 
\end{tikzpicture}

 \end{minipage}
 \caption{Some orbits  of the switched server system with the parameters of Corollary \ref{cor2}.}\label{fig2}
 \end{figure}
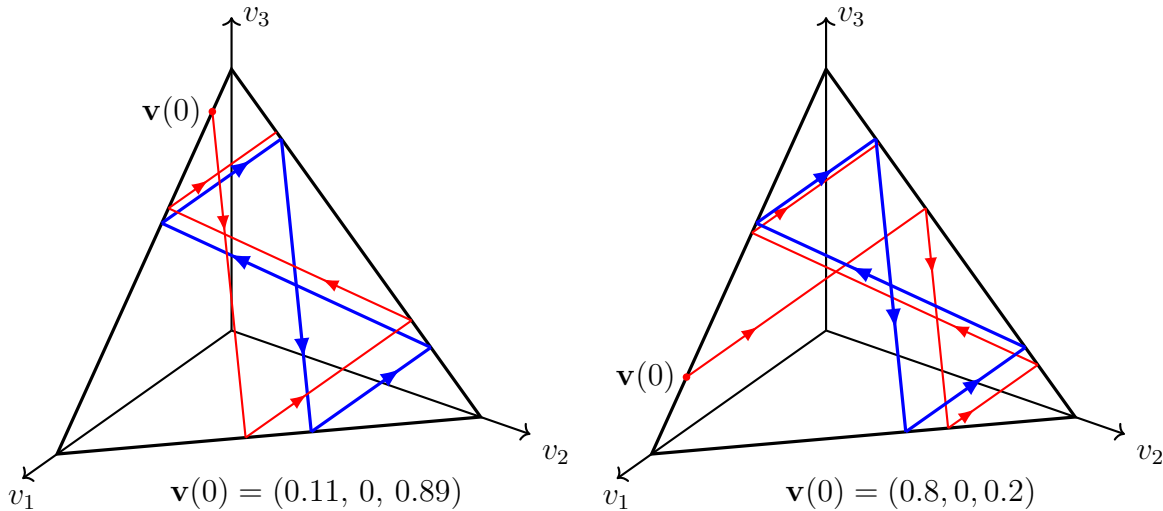

 In Figure \ref{fig2} we display the result of some computational simulations of the switched server system considering the parameter values of Corollary \ref{cor2}. The results are in agreement with the claim of Corollary \ref{cor2} since the orbits $t\in [0,\infty)\mapsto  \mathbf{v}(t)$ with $\mathbf{v}(0)=(0.11, \, 0, \, 0.89)$ and $\mathbf{v}(0)=(0.8, \, 0, \, 0.2)$, drawn in red, converge to a periodic (therefore non-fractal) attractor, drawn in blue. More precisely, the $\omega$-limit set of any ${\mathbf v}(0)$ in $\{(0.11,\, 0,\, 0.89), (0.8, \, 0, \, 0.2 )\}$ under the Poincar\'e map $F:\partial\Delta\to\partial\Delta$ 
  is the finite set
 $$ \omega({\mathbf v}(0),F)=\{(0.4, \, 0.6,\, 0),  (0,\,0.8,\,0.2), (0.4,\, 0,\, 0.6), (0,\,0.2,\,0.8)\}.$$

\section{Proofs of the results}

 \begin{definition}[$f$-invariant quasi-partition]\label{iqp} Let $f:I\to I$ be a self-map of  $I$ with a finite set of discontinuities $D\subset I$. We say that
 a finite collection of pairwise disjoint open subintervals $J_1, J_2,\ldots, J_m$ of $I{\setminus}D$ is a \textit{$f$-invariant quasi-partition} if
 $I{\setminus}\bigcup_{i=1}^m J_i$ is a finite set and there exists a self-map $\tau$ of $\{1,2,\ldots,m\}$ such that
 $f(J_i)\subseteq J_{\tau(i)}$ for all $1\le i\le m$.
 \end{definition}
 
 \begin{lemma}\label{lem2.4} Assume that there exists a $f$-invariant quasi-partition for a piecewise \linebreak $\lambda$-affine contraction  $f:I\to I$. Then the set $\omega(f)$, the global attractor of $f$, is finite.
  \end{lemma}
  \begin{proof} Let $J_i=(a_{i-1},a_i)$, $1\le i \le m$, be a $f$-invariant quasi-partition of $I$ and $\tau$ be the self-map of $\{1,2,\ldots,m\}$ as in Definition \ref{iqp}. Let $P\subseteq \{1,2,\ldots,m\}$ denote the set of periodic points of $\tau$ and let $f^{(0)}$ denote the identity map of $I$. Let
  \begin{equation}\label{FG}
  F = \bigcap_{k\ge 0} \bigcup_{i\in P} \overline{f^k(J_i)}, \quad G= [0,1]{\big \backslash} \bigcup_{i=1}^m J_i.
  \end{equation}
  We will split the proof in three claims:\\
     
  \noindent Claim A. The set $\omega(f)$ is contained in $F\cup G$.\\    
  
  In fact, let $x$ be in $I$. Then, either $\omega(x,f)\subset G$ or  there exist $i$ in $P$ and a non-negative integer $k$ such that
  $f^k(x)$ is in $J_i$. Then, by (\ref{olset}) and (\ref{FG}), we have that $\omega(x,f)=\omega\big(f^k(x),f\big)\subseteq \overline{F}$. Since $$
  \bigcup_{i\in P} \overline{f^{k+1}(J_i)}\subseteq \bigcup_{i\in P} \overline{f^{k}(J_{\tau(i)})}= \bigcup_{i\in P} \overline{f^k(J_i)},$$
  we have that the  sequence of sets $ \left\{\bigcup_{i\in P} \overline{f^k(J_i)}\right\}_{k\ge 0}$ in $(\ref{FG})$ is a nested sequence of compact sets, thus $F$ is a compact non-empty set. Therefore, $\overline{F}=F$, implying that $\omega(x,f)$ is in $F$. In either case, $\omega(x,f)\in F\cup G$.
     \\

  \noindent Claim B. We have $F = \bigcap_{j\ge 0} \bigcup_{i\in P} \overline{f^{j q}(J_i)}$, where    
  the positive integer $q$  
  is such that $\tau^{q}$ is the identity map on $P$.\\
  
  It follows from the fact that the sequence of sets $ \left\{\bigcup_{i\in P} \overline{f^k(J_i)}\right\}_{k\ge 0}$ in $(\ref{FG})$ is nested.\\
   
  \noindent Claim C. The set $F\cup G$ is a finite set.\\     
  
The set  $G$ is finite  
because $J_1, J_2, \ldots, J_m$ is a $f$-invariant quasi-partition. Let us prove now that $F$ is finite. By Definition \ref{iqp}, $J_i\cap D$ is empty for all $i$,   
where $D$ is the set of discontinuities of $f$. Yet, since $f^k(J_i)\subseteq J_{\tau^k(i)}$, we have that $f^k(J_i)\cap D$ is empty for all $i$.  
Hence, because $f$ is a piecewise $\lambda$-affine map continuous on $I{\setminus}D$ we conclude that ${f^k(J_i)}$ is an interval of length less than or equal to $\lambda^k$ for all $k\ge 0$.  Since
   $$f^{(j+1)q}(J_i)\subseteq f^{jq}(J_{\tau^q(i)})=f^{jq}(J_i),$$
   we have that $\left\{\overline{f^{jq}(J_{i})} \right\}_{j\ge 0}$ is a nested sequence of compact intervals. Therefore, there exists a point $c_i$ in $[0,1]$ such that   
$\{c_i\}=\bigcap_{j\ge 0} \overline{f^{jq}(J_i)}$. Moreover, as already pointed out, $\overline{f^{jq}(J_i)}$ is an interval   
of length less or equal to $\lambda^{jq}$ containing $c_i$. By Claim B, we conclude that $F=\{c_1,\ldots,c_m\}$, where it may occur that $c_i=c_j$ for some $i\neq j$. In this way, $F$ is a finite set.
  
     \end{proof}

\begin{proof}[Proof of Theorem \ref{thm1}] We assume  all the hypotheses and notation in the statement of Theorem \ref{thm1}. We begin by observing that the map $f$ is injective. In fact, suppose that
$f(x)=f(y)$. Then $\vert x-y\vert = \left\vert\alpha_j-\alpha_i\right\vert.$
 In particular, $\vert x- y\vert$ is an integer  
 and $0\le \vert x-y\vert <1$, that is, $x=y$. In this way, if $x$ is in $[0,1)$ and $k$ is an integer,    
 then, by the injectivity of $f$, the set $f^{-k}(\{x\})=\{y\in I: f^k(y)=x\}$ is either 
 a one-point-set or empty.    
 We  affirm that:\\

\noindent Claim D.  For each $1\le i\le n-1$, there exists a positive integer $k_i$    
such that $$\bigcup_{k=0}^\infty f^{-k}(\{x_i\})=\bigcup_{k=0}^{k_i} f^{-k}(\{x_i\}).$$

First  assume that Claim D holds. 
Without any restriction, we assume that $k_i$ is minimal for the property of the claim.
Then, for each $0\le k\le k_i$, the set $f^{-k}(\{x_i\})$ has a unique element, which we denote by
$f^{-k}(x_i)$, thus $\{f^{-k}(x_i)\}=f^{-k}(\{x_i\})$. Let
$$ H= \left\{f^{-k}(x_i):0\le k\le k_i\,\,\textrm{and}\,\, 1\le i\le n-1\right\}.
$$
Let $J_1,J_2,\ldots,J_m$ be the connected components (open intervals) of $(0,1){\backslash} H$. It is clear that no point of $\bigcup_{s=1}^m J_s$ can be mapped into a point of $H$. In particular, $f(J_s)$ is an interval that contains no point of $H$, i.e. $f(J_s)\subseteq J_{\tau(s)}$ for some $\tau(s)$ in $\{1,2,\ldots,m\}$. In this way, $J_1,J_2,\ldots,J_m$ is an $f$-invariant quasi-partition. Applying Lemma \ref{lem2.4} we conclude that Theorem \ref{thm1} holds provided we prove Claim D.

Now it remains to prove Claim D. Let $i$ be in $\{1, 2,\ldots , n-1\}$. If $f^{-k}(\{x_i\})$ is the empty-set for some
positive integer $k$, then Claim $D$ holds for such $i$. Otherwise, by the injectivity of $f$, we have  that $f^{-k}(\{x_i\})$ is a one-point-set for all positive   
integers $k$.   
As before, let $f^{-k}(x_i)$ denote the unique element of the set $f^{-k}(\{x_i\})$.
Then $f^{-k}(x_i)$ is in $f(I)$ for all integers $k\ge 0$.

In the sequel, let $k\ge 0$ be fixed. Since $f^{-k}(x_i)$ is in $f(I)$, we know that there exist 
$j$ in $\{1, \ldots n\}$ and $y_i$ in $[x_{j-1},x_j)$ such that   
\begin{equation}\label{fab1}
f^{-k}(x_i)= f(y_i)= \frac{1}{\beta} y_i + \frac{\alpha_j-1}{\beta}\cdot
\end{equation}
In this way,  we have that 
\begin{equation}\label{vd}    f^{-k}(x_i)\in \left[ \frac{\alpha_j-1}{\beta}, \frac{\alpha_j}{\beta}\right) \quad\textrm{and}\quad f^{-(k+1)}(x_i)=   y_i = \beta f^{-k}(x_i) + 1 - \alpha_j.
\end{equation}
We may rewrite \eqref{vd} as
\begin{equation}\label{eq4}  f^{-(k+1)} (x_i) = T_{\beta} \big(f^{-k}(x_i)\big),\end{equation}
where $T_{\beta}:I\to I$ is the $\beta$-transformation $x\mapsto \{\beta x\}$ or, equivalently,
\begin{equation}\label{btransformation} 
T_{\beta}(x)=\beta x +1 - r\quad\textrm{for all}\quad x \in \left[\frac{r-1}{\beta}, \frac{r}{\beta} \right)\quad\textrm{and}\quad r\in \{1,2,\ldots,\beta\}.
\end{equation}
Since $k\ge 0$ is arbitrary in \eqref{eq4}, by induction on $k\ge 0$, we obtain that
\begin{equation}\label{contr1}
\{T_{\beta}^k(x_i): k\ge 0\}=\left\{ f^{-k}(x_i): k\ge 0\right\}\subset f(I).
\end{equation}
Now we have two cases to consider. \\

Case (a): $x_i\in\mathbb{Q}$

In this case, the $\beta$-ary expansion $w$ of $x_i$ is ultimately periodic meaning that the $T_{\beta}$-orbit of $x_i$ is finite. Then, by (\ref{contr1}), Claim D holds for such $i$.\\

Case (b): $x_i\in S_{\beta,n}$

In this case, since $f$ is a piecewise $\frac{1}{\beta}$-affine contraction, we have that $I{\setminus} f(I)$ contains an open interval $J$ of length $\vert J\vert=(1-\beta^{-1})/(n+1)$. By the definition of the integer $\ell$ in  Section 2, we have that $\beta^{-\ell}< \frac12 \vert J\vert$. Hence, there exists $0\le p\le \beta^{\ell}-1$ such that
$\left[ \dfrac{p}{\beta^{\ell}}, \dfrac{p+1}{\beta^\ell} \right]\subset J$. Since $x_i$ is in $S_{\beta,n}$,  
the $T_{\beta}$-orbit of $x_i$ visits the intervals $\left[ \dfrac{r}{\beta^{\ell}}, \dfrac{r+1}{\beta^\ell} \right]$ for all $0\le r\le \beta^\ell-1$.  In particular, there exists a positive integer $k'$  
such that $T_{\beta}^{k'}(x_i)$ is in $J\subset I{\setminus} f(I)$,    
which contradicts $(\ref{contr1})$.
\end{proof} 

\begin{lemma}\label{2TT} Given an integer $\beta\ge 2$, 
let $T_{\beta}:I\to I$ and $T_{-\beta}:I\to I$ be the transformations defined by $T_{\beta}(x)=\{\beta x\}$ and 
$T_{-\beta}(x)=\{-\beta x\}$, respectively.  
Then $\big(T_{-\beta}\big)^2(x)=T_{\beta^2}(x)$ for all $x\in I$.
\end{lemma}
\begin{proof} It is clear that the equality holds for $x=0$. Otherwise, we have two cases to consider. Case $(i):$ there exist $r,p\in \{1,2,\ldots,\beta\}$ such that
$$x\in \left(\dfrac{r-1}{\beta}+\dfrac{p-1}{\beta^2},\dfrac{r-1}{\beta} + \dfrac{p}{\beta^2} \right)\subset\left( \dfrac{r-1}{\beta}, \dfrac{r}{\beta} \right).$$
Hence,
$$ T_{-\beta}(x)=-\beta x +  r,\quad\textrm{thus}\quad T_{-\beta}(x)\in \left(\dfrac{\beta-p}{\beta}, \dfrac{\beta-p+1}{\beta}
\right).$$
Applying the same reasoning to the point $T_{-\beta}(x)$ yields
$$ \big(T_{-\beta}\big)^2(x)=-\beta (-\beta x +r) +\beta-p+1 =\beta^2 x -\beta r +\beta-p+1.
$$
Likewise, since
$$x\in \left(\dfrac{\beta(r-1)+p-1}{\beta^2},\dfrac{\beta(r-1)+p}{\beta^2}  \right),$$
we have that
$$ T_{\beta^2}(x)=\beta^2 x + 1-\left(\beta (r-1) + p \right)=\beta^2 x -\beta r +\beta -p+1.
$$
Therefore, 
 \begin{equation}\label{2ig}
 \big(T_{-\beta}\big)^2(x)= T_{\beta^2}(x).
 \end{equation}
Case $(ii):$ $x=\dfrac{r-1}{\beta}+\dfrac{p}{\beta^2}$ for some $r\in \{1,2,\ldots,\beta\}$ and $p\in \{1,2,\ldots,\beta-1\}$.
 
In this case, we have that $$\big(T_{-\beta}\big)^2(x)=T_{-\beta}(-\beta x+r)=T_{-\beta}\left(\dfrac{\beta-p}{\beta}\right)=-\beta \left (\dfrac{\beta-p}{\beta}\right)+\beta-p=0=T_{\beta^2}(x).$$
\end{proof}
\begin{proof}[Proof of Theorem \ref{thm2}] The proof is exactly the same as the proof of Theorem \ref{thm1} until before equation (\ref{fab1}). From that point on, the proof should be modified as follows.

In the sequel, let $k\ge 0$ be fixed. Since $f^{-k}(x_i)$ is in $f(I)$, we know that there exist 
$j$ in $\{1, \ldots n\}$ and $y_i$ in $[x_{j-1},x_j)$ such that     
\begin{equation}\label{220}
f^{-k}(x_i)= f(y_i)= -\frac{1}{\beta} y_i + \frac{\alpha_j}{\beta}\cdot
\end{equation}
We affirm that $y_i\neq 0$. In fact, if $y_i=0$, then $0=y_i=f^{-k+1}(x_i)\in f(I)$, which contradicts the fact that $0\not\in f(I)$. Hence, $0<y_i<1$. By (\ref{220}), we have that 
\begin{equation}\label{221}  f^{-k}(x_i)\in \left( \frac{\alpha_j-1}{\beta}, \frac{\alpha_j}{\beta}\right) \quad\textrm{and}\quad f^{-(k+1)}(x_i)=   y_i = -\beta f^{-k}(x_i)  + \alpha_j.
\end{equation}
We may rewrite (\ref{221}) as
\begin{equation}\label{224}  f^{-(k+1)} (x_i) = T_{-\beta} \big(f^{-k}(x_i)\big),
\end{equation}
where $T_{-\beta}:I\to I$ is the transformation $x\mapsto \{-\beta x\}$, or equivalently, $T_{-\beta}(0)=0$ and   
$$ T_{-\beta}(x)=-\beta x + r\quad\textrm{for all}\quad x \in \left(\frac{r-1}{\beta}, \frac{r}{\beta} \right]\cap I\quad\textrm{and}\quad r\in \{1,2,\ldots,\beta\}.
$$
Since $k\ge 0$ is arbitrary in \eqref{224}, by induction on $k\ge 0$, we obtain that
\begin{equation}\label{225}
 \{T_{-\beta}^k(x_i): k\ge 0\}=\left\{ f^{-k}(x_i): k\ge 0\right\}\subset f(I).
\end{equation}
By Lemma \ref{2TT} and by (\ref{225}), we reach
\begin{equation}\label{226}
  \left\{\big(T_{\beta^2}\big)^k(x_i): k\ge 0\right\}=  \left\{\big(T_{-\beta}\big)^{2k}(x_i): k\ge 0\right\}\subset f(I).
\end{equation}
Now we have two cases to consider. \\

Case (a): $x_i\in\mathbb{Q}$.

In this case, the $\beta^2$-ary expansion $w$ of $x_i$ is ultimately periodic meaning that the $T_{\beta^2}$-orbit of $x_i$ is finite. Then, by (\ref{226}), Claim D holds for such $i$.
\\

Case (b): $x_i\in S_{\beta,n}'$

In this case, since $f$ is a piecewise $-\frac{1}{\beta}$-affine contraction, we have that $I{\setminus} f(I)$ contains an open interval $J$ of length $\vert J\vert=(1-\beta^{-1})/(n+1)$. By the definition of the integer $\ell'$ in  Section $2$, we have that $\beta^{-2\ell'}< \frac12 \vert J\vert$. Hence, there exists $0\le p\le \beta^{2\ell'}-1$ such that
$\left[ \dfrac{p}{\beta^{2\ell'}}, \dfrac{p+1}{\beta^{2\ell'}} \right]\subset J$. Since $x_i$ is in $S'_{\beta,n}$,  
the $T_{\beta^2}$-orbit of $x_i$ visits the intervals $\left[ \dfrac{r}{\beta^{2\ell'}}, \dfrac{r+1}{\beta^{2\ell'}} \right]$ for all $0\le r\le \beta^{2\ell'}-1$.  In particular, there exists a positive integer $k'$  
such that $T_{\beta^2}^{k'}(x_i)$ is in $J\subset I{\setminus} f(I)$,    
which contradicts $(\ref{226})$.
\end{proof}

\begin{proof}[Proof of Corollary \ref{cor1}] Let $(d_1,d_2,d_3)$ be as in the statement of Corollary \ref{cor1}. The  Poincar\'e map $F:\partial\Delta\to\partial\Delta$ of any switched server system with parameters $d_{ij}$ satisfying \eqref{d1d2d3} is topologically conjugate to the map to the map $f$ defined by $(\ref{fd1d2d3})$. Moreover, given $1\le i\le 3$, either $x_i\in\mathbb{Q}$ or the $4$-ary expansion of $x_i$ contain all words of length $3$ over $\{0,1,2,3\}$ and thus $x_i\in S_{2,4}'$. By Theorem \ref{thm2}, the map $f$ (and therefore $F$) has a global finite attractor, which concludes the proof.
\end{proof}

\begin{proof}[Proof of Corollary \ref{cor2}] The base-$4$ Champernowne constant $c$ is a normal number in base $4$. Since the $4$-ary expansions of $x_1$, $x_2$ and $x_3$ equal the $4$-ary expansion of $c$, up to finitely many digits, we have that $x_1$, $x_2$ and $x_3$ are normal numbers in base $4$.  Now the result follows from Corollary \ref{cor1}.
\end{proof}

\noindent\textbf{Acknowledgments.} Part of this work was carried out while the first named author had a postdoctoral position in the University of S\~ao Paulo at Ribeir\~ao Preto. He is very thankful for the excellent working conditions there. The third named author was partially supported by grant {\#}2019/10269-3 S\~ao Paulo Research Foundation (FAPESP).

\end{document}